\scrollmode
\documentclass[12pt]{amsart}

\usepackage{amsfonts, amssymb, amscd}
\usepackage{verbatim}
\usepackage{eucal}
\usepackage{amssymb}
\usepackage{mathrsfs}
\usepackage{graphicx}
\usepackage{psfrag}
\usepackage{cite}
\usepackage{upref}
\usepackage{url}
\def\deg        {{\rm deg}}
\def\ZZ         {{\mathbb Z}}
\def\QQ         {{\mathbb Q}}
\def\PP         {{\mathbb P}}
\def\RR         {{\mathbb R}}
\def\FF         {{\mathbb F}}
\def\CC         {{\mathbb C}}
\def\HH         {{\mathfrak H}}

\def\ii         {{\rm i}}

\renewcommand{\mod}{\bmod}
\newcommand\draft[1]    {{}}
\newcommand{\OOO}{\mathscr O}
\newcommand{\abcd}[4]{\left(
        \begin{smallmatrix}#1&#2\\#3&#4\end{smallmatrix}\right)}

\newcommand{\tp}[1]{#1>\!\!>0}
\newcommand{\fa}{\mathfrak{a}}
\newcommand{\fA}{\mathfrak{A}}

\DeclareMathOperator{\diag}{diag}
\DeclareMathOperator{\Sym}{Sym}
\DeclareMathOperator{\SL}{SL}
\DeclareMathOperator{\GL}{GL}
\DeclareMathOperator{\PSL}{PSL}
\DeclareMathOperator{\Tr}{Tr}
\DeclareMathOperator{\Norm}{N}
\DeclareMathOperator{\vol}{vol}
\DeclareMathOperator{\sign}{sgn}
\DeclareMathOperator{\ord}{ord}
\DeclareMathOperator{\Gal}{Gal}

\newtheorem{lemma}{Lemma}[section]
\newtheorem{theorem}[lemma]{Theorem}
\newtheorem{corollary}[lemma]{Corollary} 
\newtheorem{proposition}[lemma]{Proposition}
\theoremstyle{definition}
\newtheorem{definition}[lemma]{Definition}

\newtheorem{remark}[lemma]{Remark}
\theoremstyle{remark}
\newtheorem*{proof*}{Proof}
\numberwithin{equation}{section}

\newenvironment{proofof}[1]{\noindent{\emph{Proof of {#1}.}}}{\qed\vspace{3ex}}

\begin{document}
\title{On Hilbert modular threefolds
of discriminant 49} 
 
\author{Lev A.  Borisov}
\address{Department of Mathematics,
Rutgers University,
110 Frelinghuysen Rd,
Piscataway, NJ 08854, USA}
\email{borisov@math.rutgers.edu}

\author{Paul E. Gunnells}
\address{Department of Mathematics and Statistics, University of
Massachusetts Amherst, Amherst, MA 01003, USA}
\email{gunnells@math.umass.edu}

\dedicatory{To Don Zagier, on the occasion of his 60th birthday.}
\date{August 8, 2011}


\thanks{The authors were partially supported by the NSF, through
grants DMS--1003445 (LB) and DMS--0801214 (PG)}

\begin{abstract}
Let $K$ be the totally real cubic field of discriminant $49$, let $\OOO$
be its ring of integers, and let $p\subset \OOO $ be the prime over $7$.
Let $\Gamma (p)\subset \Gamma = SL_{2} (\OOO)$ be the principal
congruence subgroup of level $p$.  This paper investigates the
geometry of the Hilbert modular threefold attached to $\Gamma (p) $
and some related varieties.  In particular, we discover an octic in
$\PP^3$ with $84$ isolated singular points of type $A_2$.
\end{abstract}

\maketitle
\section{Introduction}

Let $K$ be the totally real cubic field of discriminant $49$, let $\OOO$
be its ring of integers, and let $p\subset \OOO $ be the prime over $7$.
Let $\Gamma (p)\subset \Gamma = SL_{2} (\OOO)$ be the principal
congruence subgroup of level $p$.  This paper investigates the
geometry of the Hilbert modular threefold $X^{\circ} = \Gamma
(p)\backslash \HH^{3}$ and some related varieties:

\begin{enumerate}
\item Let $X$ be the minimal compactification of $X^{\circ}$, and let
$X_{ch}$ be the singular toroidal compactification built using the
fans determined by taking the cones on the faces of the convex hulls
of the totally positive lattice points in the cusp data.  Then
$X_{ch}$ is the canonical model of $X$ (Theorem \ref{canmodel}).
\item We construct parallel weight $1$ Eisenstein series $F_{0},
F_{1}, F_{2}, F_{4}$ and a parallel weight $2$ Eisenstein series
$E_{2}$ that generate the ring of symmetric Hilbert modular forms of
level $p$ and parallel weight (i.e., the subring of the parallel
weight Hilbert modular forms invariant under the action of the Galois
group $G = G (F/\QQ) \simeq \ZZ /3\ZZ$) (Theorem \ref{mainE}).
\item There is a weighted homogeneous polynomial $P$ of degree $8$
with $42$ terms such that $P (F_{0}, F_{1}, F_{2}, F_{4}, E_{2})=0$.
This polynomial generates the ideal of relations on the $F_{i}$ and
$E_{2}$, and the symmetric Hilbert modular threefold $X_{Gal} = X/G$
is the hypersurface cut out by $P=0$ in the weighted projective space
$\PP (1,1,1,1,2)$ (Theorem \ref{symhm3fd}).
\item Let $Q$ be the polynomial obtained from $P$ by setting the
weight $2$ variable to zero.  Then $Q$ has $24$ terms and defines a
degree $8$ hypersurface in $\PP^{3}$ with singular locus being $84$
quotient singularities of type $A_{2}$ (Proposition \ref{prop:octic}).
\end{enumerate}

These results can be considered part of the venerable tradition of
writing explicit equations for modular varieties, a tradition
including (i) the Klein quartic, which presents the modular curve $X
(7)$ as an explicit quartic in $\PP^{3}$, (ii) the Igusa quartic,
which is the minimal compactification of the Siegel modular threefold
of level $2$ \cite{hunt.nice, janus}, and (iii) the Burkhardt quartic,
which is a three-dimensional ball quotient \cite{janus, hunt.nice}.
Our results are in the spirit of results of van der Geer,
Hirzebruch, van de Ven, and Zagier for Hilbert modular surfaces
\cite{vdgz, vdg.minimal, vdg6, hirz.bonn, hirz.5, h.vdv, vdg.vdv}, and
use many of the same techniques: toroidal compactifications, the
Shimizu trace formula \cite{shim}, the action of the finite group
$\Gamma /\Gamma (p) \simeq SL_{2} (\FF_{7})$, and explicit
construction of modular forms.

We now give an overview of the paper.  Section \ref{s:F} collects
basic facts about our field $K$, and section \ref{s:resolution}
describes the toroidal resolutions $X_{sm}$ and $X_{ch}$ we consider.
In section \ref{s:intersections} we compute intersection numbers of
exceptional divisors on $X_{sm}$ and $X_{ch}$.  The modular forms we
need are constructed in section \ref{s:eisenstein}, and in section
\ref{s:relation} we compute the polynomial relation they satisfy.
Section \ref{s:symmetric} treats the symmetric Hilbert modular
threefold, and section \ref{s:canonical} contains the proof that
$X_{ch}$ is the canonical model.  Finally, section \ref{s:octic}
describes the octic with $84$ $A_{2}$-singularities.

\section{The cubic field of discriminant $49$}\label{s:F}
In this section we summarize standard facts about the totally real
cubic field $K$ of discriminant $49$ and its ring of integers $\OOO$.

The field $K$ is the maximal totally real subfield of the cyclotomic
field $\QQ (\zeta_{7})$, where $\zeta_7=\exp(2\pi\ii/7)$.  Thus $K$ is
obtained by adjoining to $\QQ$ the number $w=\zeta_7+\zeta_7^{-1}$.
The extension $K/\QQ$ is Galois, with $\Gal (K/\QQ) \simeq \ZZ /3\ZZ$,
and a generator of the Galois group maps $w$ to $w^2-2$.  The norm and
trace of an element $r=a+bw+cw^2$ are given respectively by
\[
\Norm (r)=a^3 -a^2b + 5a^2b -2ab^2 - abc + 6ac^2 + b^3 - b^2c - 2bc^2 + c^3
\]
and 
\[
\Tr(r)=3a-b+5c.
\]
We fix an isomorphism 
\begin{equation}\label{eq:iso}
K\otimes \RR \simeq \RR^{3},
\end{equation}
and for any $a\in K$ denote its image in $\RR^{3}$ as $a\mapsto
(a_{1},a_{2},a_{3})$.  The Galois group cyclically
permutes the $a_{i}$.

The elements $\{1,w,w^{2} \}$ form a $\ZZ$-basis of $\OOO$.  Thus $\OOO$
is isomorphic to $\ZZ[w]/\langle w^3+w^2-2w-1\rangle$. It has class
number $1$, in other words it is a PID. The extension $\OOO /\ZZ$ is
ramified only over the prime $7$. The ideal $\langle 7\rangle \subset
\ZZ$ factors as $p^{3}$, where the prime $p\subset {\mathcal O}$ is
generated by the totally positive element $2-w$ of norm $7$ and trace
$7$.  The full unit group $\OOO^{\times}$ is generated by the elements
$w, w^{2}-1$.  The subgroup of units of norm $1$ is generated by the
pair $w$, $-w-1$, and the subgroup of totally positive units is
generated by $w^2$, $w^2+2w+1$.

In what follows we will also need a list of some totally positive
elements of $\mathcal O$ of small trace.  In fact we only need the
totally positive elements of traces $7$ and $14$, which are shown in
Table \ref{tracetable2}.  We used GP-Pari \cite{pari} to generate the
table.  First, we found a basis of the rank 2 $\ZZ$-module $L_{0}\subset
\OOO$ of elements of trace $0$.  Then we computed bounds for the
intersections of the translates $-nw + L_{0}$, $n=0,1,2,\dotsc $ by
the element $-w$ of trace $1$.

\begin{table}[htb]
\begin{center}
\begin{tabular}{|l||c|c|c|}
\hline
Trace $7$&$-w^2 + 4$&$-w + 2$&$w^2 + w + 1$\\
\hline
&$-5w^2 - 3w + 12$&$-3w^2 - 2w + 9$&$-2w^2 - 3w + 7$\\
&$-2w^2 + 8$&$-w^2 - w + 6$&$-2w + 4$\\
Trace $14$&$-w^2 + 2w + 7$&$w + 5$&$w^2 + 3$\\
&$2w^2 - w + 1$&$3w^2 - 2w - 1$&$w^2 + 3w + 4$\\
&$2w^2 + 2w + 2$&$3w^2 + w$&$2w^2 + 5w + 3$\\
\hline
\end{tabular}
\end{center}
\medskip
\caption{Totally positive elements of $\OOO$ of trace $7$ and $14$}\label{tracetable2}
\end{table}

\section{Resolution of the cusps}\label{s:resolution}
Let $\HH$ be the upper half plane.  We denote elements of the product
$\HH^{3}$ using multi-index notation, i.e.~given $z\in \HH^{3}$ we
write $z= (z_{1},z_{2},z_{3})$. Let $\Gamma = \Gamma (1)$ be the
Hilbert modular group $\SL (2,\OOO)$, and let $\Gamma (p) \subset
\Gamma$ be the principal congruence subgroup of matrices
$\abcd{a}{b}{c}{d}$ with $b,c\in p$ and $a, d = 1 \bmod p$.  Let
$X^{\circ }$ be the Hilbert modular threefold $\Gamma(p)\backslash
\HH^{3}$, where as usual $\Gamma (p)$ acts on $\HH^{3}$ via the three
real embeddings of $K$ in $\RR$ corresponding to \eqref{eq:iso}.  Thus
we have
\[
\Bigl(
\begin{array}{cc}
a&b\\
c&d\\
\end{array}\Bigr) \cdot z =
\Bigl(\frac{a_{1}z_{1}+b_{1}}{c_{1}z_{1}+d_{1}},
\frac{a_{2}z_{2}+b_{2}}{c_{2}z_{2}+d_{2}},
\frac{a_{3}z_{3}+b_{3}}{c_{3}z_{3}+d_{3}}\Bigr), \quad \Bigl(
\begin{array}{cc}
a&b\\
c&d\\
\end{array}\Bigr)\in \Gamma (p), z\in \HH^{3}.
\]
It is easy to see that the group $\Gamma(p)$ has no elliptic elements
(elements of finite order), so that the quotient $X^{\circ}$ is a
smooth complex threefold.  The finite group $G=\Gamma(1)/\Gamma(p)$
acts on $X$, with the center $Z (G)$ acting trivially.  The residue
field $\OOO/p $ is isomorphic to $\FF_7$, hence $G \simeq \SL(2,\FF_7)$.
The action of $G$ on $\HH^{3}$ factors through the quotient $G/Z (G)
\simeq \PSL(2,\FF_7)$, the simple group of order $168$.

Let $X$ be the minimal (Baily--Borel--Satake) compactification of
$X^{\circ}$ obtained by adjoining cusps.  Recall that $X$ is the
quotient of the partial compactification $\overline{\HH}^{3}$ obtained
by adjoining the set of cusps $\PP^{1} (K)$ to $\HH^{3}$ and endowing
the result with the Satake topology.  The variety $X$ is singular;
resolutions of the cusp singularities of Hilbert modular varieties
were described for quadratic fields by Hirzebruch \cite{hirz}, and in
general by Ehlers \cite{ehlers} (see also \cite[Appendix of
III.7]{Freitag}).  The goal of this section is to explicitly describe
the resolution of singularities of the cusps of $X$.

\begin{lemma}\label{lem:cusps}
The Hilbert modular threefold $X$ has $8$ cusps.  The group $G$ acts
transitively on the cusps of $X$.
\end{lemma}

\begin{proof}
A generalization of the standard argument for the principal congruence
subgroup of $\SL (2,\ZZ)$ (cf. \cite[Lemma 1.42]{shimura}) shows that
the cusps are in bijection with the set of nonzero points in $(\OOO
/p)^{2}$ modulo the induced action of the units.  It is not hard to
compute that the resulting set is isomorphic to $\PP^{1} (\FF_{7})$.
This proves the first statement.

For the second, one observes that action of $G$ on the cusps is
the same as the induced action of $G$ on $\PP^{1} (\FF_{7})$ under the
above bijection.  
\end{proof}

Hence Lemma \ref{lem:cusps} implies that to resolve the cusps it
suffices to look at the cusp corresponding to the image of $(\ii\infty
,\ii\infty ,\ii\infty)$ in $\overline{\HH}^3$.  Before
investigating this cusp, we first describe the cusp singularity and a
toroidal resolution of the unique cusp of $Y = \Gamma(1)\backslash
\overline{\HH}^{3}$, the minimal compactification of the Hilbert
modular threefold attached to the full group $\Gamma (1)$.

The inverse different of the ring $\OOO$ is the fractional
ideal $p^{-2}$.  Let $C\subset p^{-2}\cong \ZZ^3$ be the totally
nonnegative cone.  Namely $C$ consists of the elements of $p^{-2}$
whose images under \eqref{eq:iso} are nonnegative. Clearly, these are
just $0$ and the totally positive elements in $p^{-2}$. The group
$U\cong \ZZ^2$ of totally positive units acts on $C$. The local
(analytic) ring of the cusp for the group $\Gamma(1)$ is given by the
$U$-invariant Fourier series with exponents in $C$ that converge in
some neighborhood of the cusp (cf.~\cite[Theorem 4.1]{Freitag}).

Partial toroidal resolutions of this singularity are described by
rational polyhedral fans in the dual cone $C^*$ in the dual of
$p^{-2}$ under the trace pairing, i.e.~by fans in the totally
nonnegative cone in the lattice $\OOO$.  One needs the fan to be
invariant under the natural action of $U$ on $C^{*}$.  

The resulting resolution will be nonsingular if and only if the fan is
simplicial with each maximal cone generated by a $\ZZ$-basis of
$\OOO$. One such resolution for our singularity was constructed
explicitly by Grundman \cite{Grundman}. However, the goal there was to
compute special values of partial zeta functions using Shintani's
formula \cite{shintani}, not to construct a geometrically natural
resolution.  Hence the resolution found in \cite{Grundman} is not the
most useful for our purposes.

To describe a somewhat better resolution from our perspective, observe
that there is always a natural \emph{partial} resolution of a cusp
singularity: one begins with the convex hull $\Pi$ of the nonzero
points in $C^*\cap \OOO$, and builds the fan whose cones of maximal
dimension are the cones on the facets of $\Pi$.  For quadratic
fields, such fans lead to canonical resolutions of the cusp
singularities of Hilbert modular surfaces.  For fields of higher
degree, the resulting varieties are usually singular; indeed, the fans
one obtains need not be simplicial.

For the cusp of $Y$, we can describe this canonical fan as follows.
We start with the fundamental parallelogram of the action of $U$ as in
\cite[Figure 1]{Grundman}, but then subdivide it by the
\emph{opposite} diagonal (Figure \ref{fig1}).  Note that vertices in
Figure \ref{fig1} are labeled by units in $\OOO$.

\begin{figure}[tbh]
\psfrag{1}{$1$}
\psfrag{w2}{$w^{2}$}
\psfrag{w+2}{$w+{2}$}
\psfrag{w12}{$(w+1)^{2}$}
\begin{center}
\includegraphics[scale = .3]{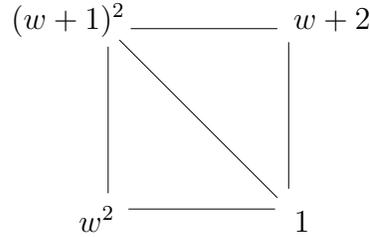}
\end{center}
\caption{The canonical partial resolution $Y_{ch}$.\label{fig1}}
\end{figure}

\begin{proposition}
The translates of two triangles of Figure \ref{fig1} by $U$ are
precisely the facets of the convex hull of the set of totally positive
integers.
\end{proposition}

\begin{proof}
It is clear that the cones spanned by these triangles cover the entire
cone $C^*$. Consequently, it is only necessary to show that each of these
triangles is indeed a face of the boundary. This means that the 
linear functions that equal $1$ on this face are bigger than $1$ on any
other totally positive element of $\OOO$.

Consider first the triangle with vertices $1$, $w^2$, and $(w+1)^{2}$.
The supporting inequality determined by this triangle on any integer
$r=a+bw+cw^2$ is $2a-b+2c\geq 2$. Let $\omega$ be the totally positive
integer $9-2w-3w^2$.  If $r$ is a totally positive element in
$\OOO$ with $2a-b+2c\leq 2$, then $\Tr(r\omega )\leq 14$. Given
that $\omega $ lies in $p$, we see that $\Tr(r\omega )$ is in $p$, so
it has to be $7$ or $14$.

According to Table \ref{tracetable2}, there are three totally positive
integers of trace $7$.  One easily sees that none are divisible by
$\omega $ in $\OOO$.  The table also shows that there are $15$
elements of trace $14$.  Exactly three of them are divisible by
$\omega $, and the ratios are $1$, $w^2$, and $(w+1)^{2}$.  Thus this
triangle is a facet of the convex hull.


The argument for the triangle with vertices $1$, $w+2$, and
$(w+1)^{2}$ is similar.  The supporting inequality on $r=a+bw+cw^2$ is
$a-b+2c\geq 1$.  If $a-b+2c\leq 1$, then $\Tr(r\eta )\leq 7$, where
$\eta = 2-w$. Again $\eta \in p$, which means $\Tr(r\eta )=7$ and
$r\eta \in\{4-w^2,2-w,1+w+w^2\}$. Consequently, $1$, $w+2$, and
$(w+1)^{2}$ are vertices of a facet of the convex hull of totally
positive elements of $\OOO$, which finishes the proof.
\end{proof}

Using the partial resolution we can construct a resolution of the cusp
singularity of $Y$.  The points $1$, $w^2$, $1+2w+w^2$ generate a
sublattice of index $2$ in $\OOO$, and thus the corresponding cone must
be subdivided.  We subdivide all translates of this cone into three
cones. Geometrically this amounts to blowing up the unique singular
point on the partial resolution.  On the other hand the points $1$,
$2+w$, $1+2w+w^2$ form a $\ZZ$-basis of $\OOO$, and thus no blowups are
needed on its translates.  The resulting smooth triangulation is shown
in Figure \ref{fig2}.

\begin{figure}[tbh]
\psfrag{1}{$1$}
\psfrag{w2}{$w^{2}$}
\psfrag{w+2}{$w+{2}$}
\psfrag{w12}{$(w+1)^{2}$}
\psfrag{ww1}{$w^{2}+w+1$}
\psfrag{a}{$\phantom{w^{2}+w+1}$}
\begin{center}
\includegraphics[scale = .3]{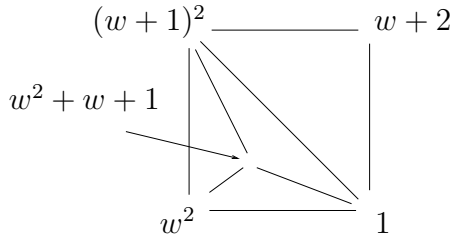}
\end{center}
\caption{The smooth resolution $Y_{sm}$.\label{fig2}}
\end{figure}

Now we want to resolve the singularities of the cusps of $X$.  The
discussion is essentially the same as above, although the lattice and
unit groups change.  The lattice
$\OOO$ needs to be replaced by the lattice $p$.  Define the
subgroup $U_{1}\subset U$ by
\begin{equation}\label{}
U_{1} = \{ u\in U\mid  u=1\bmod p\}. 
\end{equation}
This is a subgroup of index 3 in $U$.  Then we have
$$
\left(\begin{array}{cc} 
u&*\\
0&u^{-1}
\end{array}
\right)\in \Gamma (p)\quad \text{if and only if} \quad u \in  U_{1}.
$$
Moreover, recall that the ideal $p$ is generated by the totally
positive element $2-w$.  Thus the resolution of the cusp of $X$ is
basically the same as that for $Y$.  All one needs to do is to
multiply the vertices of Figures~\ref{fig1} and~\ref{fig2} by $2-w$ to
dilate the lattice $\OOO$ to $p$ (cf.~Remark \ref{nearby}), and then to take $3$ copies of the
domains in Figures~\ref{fig1} and~\ref{fig2} to reflect the smaller
group of units.  The resulting resolutions $X_{ch}$ and $X_{sm}$ are
depicted in Figure \ref{fig3}.

\begin{itemize}
\item The partial resolution $X_{ch}$, which corresponds to taking the
convex hull, is obtained by erasing the $E_{i}$ and the lines
emanating from them. The variety $X_{ch}$ is not smooth: it has 24 isolated $
\ZZ/2\ZZ$ quotient singularities.  The
exceptional divisors of $\pi_{1}\colon X_{ch}\to X$ over each cusp are
$D_1,D_2,D_3$. Since there are $8$ cusps, there are altogether $24$
exceptional divisors.
\item The full resolution $\pi_{2}\colon X_{sm}\rightarrow X$ is given
by the entire Figure \ref{fig3}. It is the blowup of $X_{ch}$ at its
$24$ $\ZZ /2\ZZ $ quotient singularities. Thus the map $\pi_{2}$ has
$48$ exceptional divisors.  Note that there is a natural map
$X_{sm}\to X_{ch}$; the exceptional divisors are $E_1,E_2,E_3$ over
each cusp, and so there are $24$ components overall.

\end{itemize}

\begin{remark}\label{nearby}
The points in Figure \ref{fig3} correspond (left-to-right and then top
to bottom) to elements
$1+2w+w^2,~2+w;~1+w+w^2,~4-w^2;~w^2,~1,~5-w-2w^2;~2-w;~-w+w^2,~3-w-w^2$
of the totally positive cone $C^*$ in $\OOO$. The corresponding
points in the cone of $p$ are given by multiplication of these points
by $2-w$.
\end{remark}

\begin{figure}[tbh]
\psfrag{e1}{$\scriptstyle {E_{1}}$}
\psfrag{e2}{$\scriptstyle {E_{2}}$}
\psfrag{e3}{$\scriptstyle {E_{3}}$}
\psfrag{d1}{$\scriptstyle {D_{1}}$}
\psfrag{d2}{$\scriptstyle {D_{2}}$}
\psfrag{d3}{$\scriptstyle {D_{3}}$}
\begin{center}
\includegraphics[scale = .4]{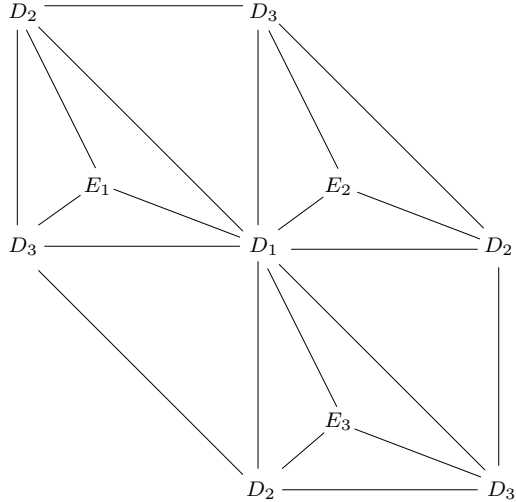}
\end{center}
\caption{Divisors on $X_{sm}$.\label{fig3}}
\end{figure}

We will introduce the following notations. Let $D$ be the sum of the
exceptional divisors of type $D_i$, and let $E$ be the sum of the
exceptional divisors of type $E_i$ on $X_{sm}$. We will also abuse
notation and use the same notation for the divisors $D_i$ on $X_{ch}$
and for their sum.  We will summarize some of the easy facts that
follow from the toric geometry.

\begin{proposition}\label{infdiv}

\phantom{$a$}
\begin{enumerate}
\item The divisor $D+E$ on $X_{sm}$ has simple normal crossings.  
\item Each
component of the divisor $E$ is isomorphic to $\PP^2$, with normal
bundle ${\mathcal O}(-2)$. 
\item The component $D_1$ of the divisor $D$ is
isomorphic to a toric surface with the fan given in the left part of
Figure \ref{fig4}, where the divisor names indicate the intersection
with other components of $D+E$.  
\item The normal bundle of $D_1$ is given
by the line bundle corresponding to the piecewise linear function
on the fan whose values at the generators of one-dimensional cones of
the fan are given in the right part of Figure \ref{fig4}.  
\item The discrepancies are given by
$$
K_{X_{ch}}=\pi_1^*K_X-D,~K_{X_{sm}}=\pi_2^*(K_{ch})+\frac 12 E = (\pi_2\circ \pi_1)^*K_X - D - E.
$$
\end{enumerate}
\end{proposition}
                        
\begin{proof}
The vertices of every triangle in Figure \ref{fig3} correspond to
three different components of $D+E$, which implies the simple normal
crossing statement. The components of $E$ are obtained by blowing up
an isolated $\ZZ/2\ZZ $-quotient singularity on $X_{ch}$; this implies
their description as a $\PP^2$ with normal bundle ${\mathcal O}(-2)$.

To find the structure of the divisor $D_1$ from Figure \ref{fig3}, we
project the surrounding vertices in Figure \ref{fig3} to a dimension
two lattice obtained by modding out the element that corresponds to
$D_1$.  Remark \ref{nearby} then leads to Figure \ref{fig4}. To
calculate the normal bundle $D_1\vert D_1$, observe that in a
neighborhood of $D_1$ the divisor corresponding to the global
linear function that calculates the first coordinate of the points in
Remark \ref{nearby} is trivial.  By subtracting it from $D_1$ we
arrive at the values in the second diagram of Figure \ref{fig4}.

The calculation of discrepancies is standard and is left to the reader.
\end{proof}

\begin{figure}[tbh]
\psfrag{e1}{$\scriptstyle {E_{1}}$}
\psfrag{(-\frac 67E_2)}{$\scriptstyle {E_{2}}$}
\psfrag{e3}{$\scriptstyle {E_{3}}$}
\psfrag{d1}{$\scriptstyle {D_{1}}$}
\psfrag{d2}{$\scriptstyle {D_{2}}$}
\psfrag{d3}{$\scriptstyle {D_{3}}$}
\psfrag{0}{$\scriptstyle{0}$}
\psfrag{-1}{$\scriptstyle{-1}$}
\psfrag{-2}{$\scriptstyle{-2}$}
\psfrag{-3}{$\scriptstyle{-3}$}
\psfrag{-4}{$\scriptstyle{-4}$}
\psfrag{-5}{$\scriptstyle{-5}$}
\begin{center}
\includegraphics[scale = .3]{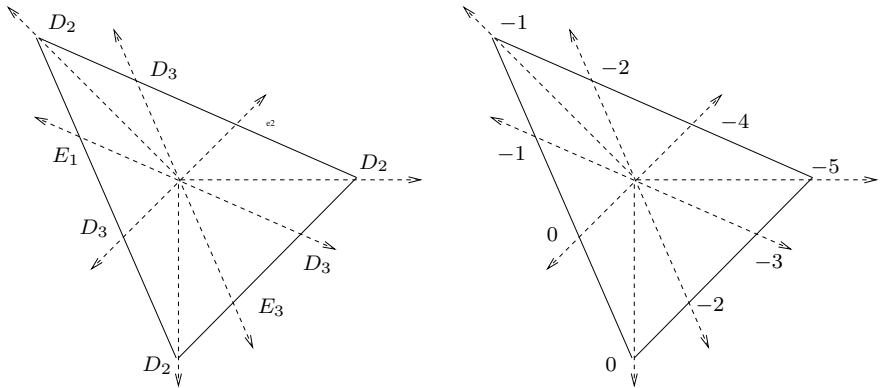}
\end{center}
\caption{The component $D_{1}$ and its normal bundle.\label{fig4}}
\end{figure}

\section{Intersection numbers}\label{s:intersections}
The goal of this section is to calculate various intersection numbers 
on $X_{sm}$ and $X_{ch}$.  For background and basic definitions for
Hilbert modular forms, we refer to \cite{Freitag}.

\begin{definition}
We will denote by $L$ the pullback from $X$ of the line bundle of a
(meromorphic) weight $(1,1,1)$ modular form. We will abuse notation
and denote by $L$ the pullback of this bundle to $X_{sm}$ and to
$X_{ch}$.
\end{definition}

Note that the global sections $H^0(X_{sm},kL)$ can
be naturally identified with weight $(k,k,k)$ modular forms
with respect to $\Gamma(p)$. Similarly, global sections $H^0(X_{sm},kL-D-E)$
are the cusp forms of weight $(k,k,k)$. 

We will need the following proposition: 
\begin{proposition}\label{canon}
\emph{\cite[Lemma 4.6]{Freitag}} The canonical class $K$ of $X_{sm}$ is $2L-D-E$.
\end{proposition}

\begin{proposition}\label{Euler}
We have 
\[
\chi(X_{sm},kL-D-E)=\chi(X_{sm},kL) = 2(k-1)^3.
\]
\end{proposition}

\begin{proof}
For even $k>2$, 
we claim 
\begin{equation}\label{eq:dimension}
\dim H^0(X_{sm},kL-D-E) = 2(k-1)^3.  
\end{equation}
This implies both statements.  Indeed, Proposition \ref{canon} and the
Kawamata vanishing theorem \cite{kawa} imply
\[
H^0(X_{sm}, kL-D-E)=\chi(X_{sm},kL-D-E).
\]
This proves the statement about $\chi(X_{sm},kL-D-E)$, since the Euler
characteristic is polynomial in $k$.  Then Serre duality implies
$\chi(kL) = -\chi((2-k)L-D-E) = -2(1-k)^3$, which proves the statement
about $\chi(X_{sm},kL)$.

To prove \eqref{eq:dimension}, we apply the trace formula computations
of Shimizu \cite{shim}.  We have for even $k>2$
\begin{equation}\label{eqn:tracefmla}
\dim H^0(X_{sm},kL-D-E) = \vol (\Gamma (p)\backslash \HH^{3})
(k-1)^{3} + e + c,
\end{equation}
where the volume is computed with respect to a suitably normalized
invariant measure and $e$ (respectively $c$) is a contribution coming
from the elliptic points of $\Gamma (p)$ (resp., the cusps of $\Gamma
(p)$).  Since $\Gamma (p)$ is torsion-free, we have $e=0$.
Furthermore $c$ is essentially the sum of some special values of
Hecke--Shimizu $L$-functions, each of which is attached to a pair
$(M,V)$, where $M\subset K\otimes \RR \simeq \RR^{3}$ is a rank $3$
$\ZZ$-module and $V$ is a subgroup of the units fixing $M$.  The
action of $-1$ on $M$ preserves this data and takes the special value
into its negative (this happens for any totally real field of odd
degree $>1$).  Thus $c=0$, and the only contribution comes from the
volume.  By a theorem of Siegel \cite{siegel1, siegel2} we have 
\[
\vol  (\SL (2,\OOO)\backslash \HH^{3}) = - \frac{\zeta_{K} (-1)}{4},
\]
where $\zeta_{K} (s)$ is the Dedekind zeta function for $K$.  This
special value 
be easily computed, e.g.~using techniques in \cite{shintani,
dedekind}, and we have 
\[
\vol  (\SL (2,\OOO)\backslash \HH^{3}) = \frac{1}{84}.
\]
Since the index of $\Gamma (p)$ in $\SL (2,\OOO)$ equals $\# \SL
(2,\FF_{7}) = 336$ and the center of $SL (2, \OOO)$ acts trivially, we
find 
\[
\vol (\Gamma (p)\backslash \HH^{3}) = 2,
\]
which completes the proof of the claim.
\end{proof}

\begin{proposition}\label{intnumbers}
We have the following intersection numbers on $X_{sm}$:
$$
L^3=12,~(K-\frac 12 E)^3 =36.
$$
\end{proposition}

\begin{proof}
The Riemann-Roch formula for $kL$ and Proposition \ref{Euler} give
$2(k-1)^3= \frac 16k^3L^3+\cdots$, which implies $L^3=12$.  To
calculate $(K-\frac 12E)^3$ we will write it as $(K-\frac
12E)^2(2L-\frac 32E-D)$.  Since $L$ restricts trivially to $D$ and
$E$, we have $(K-\frac 12E)^2L=4L^3=48$. Since $(K-\frac 12E)$ is the
pullback of a $\QQ$-Cartier divisor $K_{ch}$ from $X_{ch}$, we have
$(K-\frac 12E)^2E = 0$. So it remains to calculate $(K-\frac
12E)^2D=24(2L-\frac 32E-D)^2D_1$. Since we know the normal bundle of
$D_1$ from Proposition \ref{infdiv}, we conclude that the restriction
of $2L-\frac 32E-D$ to $D_1$ is given by the piecewise linear function
on the left of Figure \ref{fig5}.  The self-intersection is then an
easy calculation in toric geometry.  We get $0-2(\frac 12)^2+2(-\frac
12)(-1)-2(-1)^2+2(-1)(-1)-(-1)^2 +2(-1)(\frac 12)-2(\frac
12)^2+2(\frac 12)2-2(2)^2+2(2)(4)-4^2 +2(\frac 52)4-2(\frac
52)^2+2\frac 52(1) -2(1)^2$, which equals
$\frac 52$. This gives $(K-\frac 12E)^3=96-24(\frac 52)=36$.

\begin{figure}[tbh]
\psfrag{0}{$\scriptstyle{0}$}
\psfrag{-1}{$\scriptstyle{-1}$}
\psfrag{1}{$\scriptstyle{1}$}
\psfrag{-2}{$\scriptstyle{-2}$}
\psfrag{2}{$\scriptstyle{2}$}
\psfrag{-3}{$\scriptstyle{-3}$}
\psfrag{-4}{$\scriptstyle{-4}$}
\psfrag{4}{$\scriptstyle{4}$}
\psfrag{1h}{$\scriptstyle{1/2}$}
\psfrag{5h}{$\scriptstyle{5/2}$}
\psfrag{n1h}{$\scriptstyle{-1/2}$}
\begin{center}
\includegraphics[scale = .3]{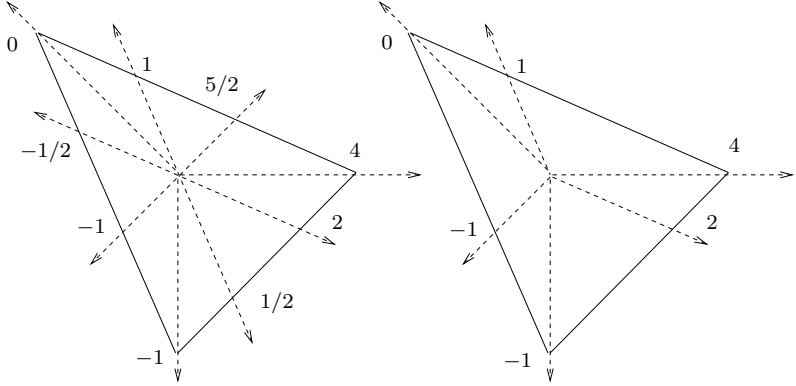}
\end{center}
\caption{The restriction of $2L-\frac 32E-D$ to $D_1$ (proof of
Proposition \ref{intnumbers}), and the toric surface $D_{1}$ on
$X_{ch}$ (proof of Lemma \ref{ampleonD1}).\label{fig5}}
\end{figure}

\end{proof}

We will now calculate the dimensions of the spaces of modular forms of
weight $(k,k,k)$ for various $k$.

\begin{proposition}\label{zeroorfour}
$\phantom{a}$
\begin{enumerate}
\item For $k\geq 3$ we have $\dim H^0(X_{sm}, kL-D-E)=2(k-1)^3$ 
and $\dim H^0(X_{sm}, kL)=2(k-1)^3+8$.
\item For $k=2$ we have $\dim H^0(X_{sm}, 2L-D-E)=3$ and $\dim
H^0(X_{sm}, 2L)=11$.
\item For $k=1$ we have $\dim H^0(X_{sm}, L-D-E)=0$. The dimension of
$H^0(X_{sm},L)$ is either $0$ or $4$.
\end{enumerate}
\end{proposition}

\begin{proof}
The statement for $k\geq 3$ about cusp forms follows from the proof of
Proposition \ref{Euler}.  Similarly, for $k=2$, the dimension of
$H^0(X_{sm},2L-D-E)$ can be computed using Riemann-Roch
(cf.~\cite[Theorem II.4.8]{Freitag}); the computations are similar to
those of the proof of Proposition \ref{Euler}.  For even $k\geq 2$
these computations also give the dimension of $H^0(X_{sm},kL)$, since
by Lemma \ref{lem:cusps} the number of cusps of $X_{sm}$ is $8$.

For any $k\geq 3$ Kawamata vanishing implies there is a short exact sequence 
$$
0\to  H^0(X_{sm},kL-D-E)\to   H^0(X_{sm},kL)\to H^0(X_{sm},{\mathcal O}_{D\cup E}(kL))\to 0.
$$
We have ${\mathcal O}_{D\cup E}(kL) = {\mathcal O}_{D\cup E}$.  Hence
the dimension of $H^{0}(X_{sm}, {\mathcal O}_{D\cup E}(kL))$ is the
number of connected components of $D\cup E$, which is 8.  This
completes the proof of (1) and (2).

We turn now to $k=1$.  Let us prove that $H^0(X_{sm},L-D-E)=0$. The
space $H^0(X_{sm},L-D-E)=0$ is a representation of the group
$G\simeq\SL(2,\FF_{7})$ such that the central involution acts by $-1$.
By investigating the character table of $G$ \cite{frobenius,schur}, we
see that the smallest dimension of such a representation is $4$. If
there were a nonzero element in $H^0(X_{sm},L-D-E)=0$, a
multiplication by it shows $\dim H^0(X_{sm},L-D-E)\leq \dim
H^0(2L-D-E)= 3$, which is a contradiction.

Finally, the space $H^0(X_{sm},L)$ is also a representation of $G$
with central involution acting by $-1$. The dimensions of the
irreducible representations of $G$ with this property are $4$, $6$ and
$8$.  There is an injection of $G$-representations $H^0(X_{sm},L)\to
H^0(X_{sm}{\mathcal O}_{D\cup E}(L))$, where the latter has dimension
$8$. This shows that $\dim H^0(X_{sm},L)$ cannot be $6$. Suppose that
$\dim H^0(X_{sm},L)=8$.  Since $\dim H^0(X_{sm},2L)=11$, the kernel of
the multiplication map $H^0(X_{sm},L)^{\otimes 2}\to H^0(X_{sm},2L)$
has codimension at most $11$. Consequently, it contains a nonzero
decomposable tensor, since the dimension of the image of the Segre
embedding of $\PP H^0(X_{sm},L)\times \PP H^0(X_{sm},L)$ into
$\PP(H^0(X_{sm},L)^{\otimes 2})$ is $14$. However, this implies a
product of two non-zero sections of $H^0(X_{sm},L)$ is a non-zero
section of $H^0(X_{sm},2L)$, contradiction.  This leaves $0$ and $4$
as the only possibilities for $\dim H^0(X_{sm},L)$.
\end{proof}

\begin{remark}
We will later show in Corollary \ref{isfour} 
that $\dim H^0(X_{sm},L)=4$ by explicitly exhibiting
a basis of Eisenstein series.
\end{remark}

\section{Eisenstein series}\label{s:eisenstein}
In this section we use Eisenstein series to construct explicitly some
modular forms of weights $(1,1,1)$ and $(2,2,2)$ for the group
$\Gamma(p)$.  Recall that the class number of $K$ is one, and thus
every ideal of $\OOO$ is principal. 

\begin{definition}
For an ideal $(c) \subset {\OOO}$ we define $s(c)=0$ if $p|c$. 
Otherwise, we define $s(c)=\sign(\Norm(c_1))$ where $c_1$ is a generator of 
$(c)$ satisfying $c_1=1\bmod p$.
\end{definition}

We note that $s(c)$ is well-defined, since different generators
$c_1,c_2$ of $(c)$ that are equal to $1\bmod p$ differ by a unit equal
to $1\bmod p$, and all such units have positive norm. Alternatively,
one could define $s(c)$ for $(c,p)=1$ as the product of
$\sign(\Norm(c))$ and the quadratic residue symbol $\big(\frac
cp\big)\in \{\pm 1 \}$ (the latter equals $1$ if and only if the image
of $c$ is a square in $\OOO /p \simeq \FF_{7}$).

\begin{definition}\label{defFi}
For ${z}=(z_1,z_2,z_3)\in \HH^3$ and $a\in K$, we 
define $\Tr(az)=a_1z_1+a_2z_2+a_3z_3$. Then for $i\in \{0,1,2,4\}$ we define
\begin{equation}\label{eq:Fi.expansion}
F_i({z}) = c_i+\sum_{\substack{a\in {\OOO}, \tp{a}\\a =i\bmod p}} \exp(2\pi\ii \Tr(az)/7)
\Big( \sum_{(c)|(a)} s(c)\Big)
\end{equation}
where $c_0=1/{14}$ and $c_1=c_2=c_4=0$. We also define
$$
E_2({z})= -\frac 1{168} + \sum_{\substack{a\in p^{-2}, \tp{a}}}  \exp(2\pi\ii \Tr(az))
\Big( \sum_{(c)|(ap^{2})} |\Norm(c)|\Big)
$$
\end{definition}

Our main result in this section is given by the following theorem:
\begin{theorem}\label{mainE}
For $i\in\{0,1,2,4\}$ the series $F_i({z})$ converges to a weight
$(1,1,1)$ modular form with respect to the group $\Gamma(p)$. The
series $E_2({z})$ converges to a weight $(2,2,2)$ modular form
with respect to the group $\Gamma(1)$.
\end{theorem}

Before proving the theorem, we need to recall some work of Yang
\cite{thyang}, who extended and corrected certain constructions of
Hecke \cite{hecke} for real quadratic fields to all totally real
fields.  First we need some notation.

Let $L$ be a CM field with maximal real subfield $K$.  Let $\OOO$ be the
ring of integers of $K$.  Suppose $[K:\QQ]=d$, and let $\chi $ be the
quadratic Hecke character attached to the extension $L/K$.  Let
$\partial_{K}$ be the different of $K$ and let $d_{L/K}$ be the
relative discriminant.  Let $(\alpha_{v}) \in \prod_{v|d_{L/K}}
F^{\times}_{v}$ be a tuple with $\ord_{v} (\alpha_{v}) = \ord_{v}
(\partial_{K})$.  Let $N$ be a square-free integral ideal of $K$
coprime to $d_{L/K}$.

\begin{theorem}\label{thm:yang}
\emph{\cite[Theorem 1.2]{thyang}}$\phantom{a}$
\begin{enumerate}
\item There is a function $E (z,s;\Phi^{\alpha ,N})\colon
\HH^{d}\times \CC \rightarrow \CC$ that as a function of $s$ is
meromorphic with possibly finitely many poles.
\item 
For $s$ fixed and away from
the poles, $E (z,s;\Phi^{\alpha ,N})$ is a (non-holomorphic) Hilbert modular form of
weight $(1,\dotsc ,1)$, level $d_{L/K}$, and character $\chi$, where
$\chi$ means
\[
\chi \colon (\OOO /d_{L/K}N)^{\times}\rightarrow (\OOO /d_{L/K})^\times
\rightarrow \{\pm 1 \}, \quad \chi (a) = \prod_{v|d_{L/K}} \chi_{v} (a).
\]
\item For $s=0$, the function $E (z,0;\Phi^{\alpha ,N})$ is a
holomorphic Hilbert modular form with Fourier expansion 
\begin{multline}\label{eq:ty.expansion}
E(z,0;\Phi^{\alpha ,N}) = (1+\varepsilon (\alpha ,N)) L (0,\chi) +\\
2^{d}\varepsilon (\alpha ,N)
\sum_{\substack{t\in\partial_{K}^{-1}N\\\tp{t}}} \delta (\alpha t)
\rho_{L/K} (t\partial_{K}N^{-1})\exp (2\pi \ii \Tr (tz)).
\end{multline}
Here $L (s,\chi)$ is the Hecke $L$-function attached to the Hecke
character $\chi$ and the quantity $\varepsilon (\alpha ,N)$ is defined
by 
\[
\varepsilon (\alpha ,N) = (-1)^{o (N)}\prod_{v|d_{L/K}}\chi_{v}
(\alpha_{v})\prod_{v|\partial_{K}, v\nmid d_{L/K}}\chi_{v} (\partial_{K})
\]
where $o (N)$ is the number of prime factors of $N$.  Furthermore
\[
\delta (\alpha t) = \prod_{v|d_{L/K}} (1+\chi_{v} (\alpha_{v}t))
\]
and 
\[
\rho_{L/K} (\fa) = \#\bigl\{\fA\subset \OOO_{L}\bigm| \Norm_{L/K}\fA = \fa \bigr\}
\]
\end{enumerate}
\end{theorem}

Now let $\Gamma_1(p)$ (resp. $\Gamma_0(p)$) be the subgroup of $\Gamma(1)$
given by the condition that its elements are upper-triangular unipotent
(resp. upper-triangular) modulo $p$.  

\begin{lemma}\label{th}
The series $F_0(z)$ converges to a modular form of weight $(1,1,1)$
for $\Gamma_1(p)$. Moreover, it has
character $\chi$ for $\Gamma_0(p)$.
\end{lemma}

\begin{proof}
Let $L=\QQ (\zeta_{7})$.  Then $L$ is a complex quadratic extension of
$K$; let $\chi$ be the corresponding Hecke character.  The relative
discriminant $d_{L/K}$ of the extension $L/K$ is the prime ideal $p$,
and the different $\partial _{K}$ is $p^{2}$.

Let $\alpha = (2-w)^{2} \in K\subset K_{p}$, and let $N = \OOO $.  We claim that 
the series $E (z) =  E(z,0;\Phi^{\alpha ,N})$ of Theorem \ref{thm:yang} is
nonzero and is a multiple of $F_{0}(z)$.  Indeed, $\varepsilon (\alpha
,N) = 1$ for this data since $\chi$ is quadratic and $o (N)=0$.  Using results of
\cite{dedekind}, we can compute that the constant term of $E (z)$ is
$L (0,\chi) = 2/7$.  Thus 
we must have
\[
E(z,0;\Phi^{\alpha ,N}) \stackrel{?}{=} 16 F_{0} (z);
\]
we will check this by comparing $q$-expansions.  By Theorem
\ref{thm:yang}, this will complete the proof.

The sum in \eqref{eq:Fi.expansion} is taken over all totally positive
$a\in p$, whereas the sum in \eqref{eq:ty.expansion} is taken over all
totally positive $t\in p^{-2}$.  To compare these $q$-expansions, we
set $a=7t$. Thus we need to check that
\[
2 \sum_{(c)|(a)} s(c) = \delta(\alpha a /7) \rho_{L/K}(a/ (2-w)), \quad a\in {p}, \tp{a}.
\]
Note that $\sum_{(c)|(a)} s(c) = \sum_{(c)|(a/ (2-w))} s(c)$, since
the extra summands in the first sum vanish.  The identities we will
prove are
\begin{align}
\sum_{(c)|(b)} s(c) &= \rho_{L/K}(b), \quad b \in \OOO,\label{eq:multiplicative}\\
\delta(\alpha a /7) &= 2 \quad \text{if $a\in {p}, \tp{a}$, and
$\rho_{L/K}(a/ (2-w)) \not = 0$.}\label{eq:delta} 
\end{align}
Since $d=3$, this will complete the proof of the theorem.

Both sides of \eqref{eq:multiplicative} are multiplicative functions,
so it suffices to check them on powers of prime ideals $P^{k}\subset \OOO$.  We have 
\[
\rho_{L/K} (P^{k})= \begin{cases}
k+1&\text{$P$ split in $L$,}\\
1&\text{$P$ inert in $L$, $k$ even,}\\
0&\text{$P$ inert in $L$, $k$ odd,}\\
1&\text{$P$ ramified in $L$}.
\end{cases} 
\]
On the other hand, from class field theory we know $s (P) = 1$
(respectively $-1$) if and only if $P$ is split in $L$ (respectively
inert in $L$).  Since the only prime that ramifies is $p$ and $s (p) =
0$, this proves \eqref{eq:multiplicative} and completes the proof of
the lemma.

\end{proof}

\begin{proofof}{Theorem \ref{mainE}}
We observe that since $p = (2-w)$,
$F_0({z'})$ is a modular form for $\Gamma(p)$, where 
$${z'} =(z_1/(2-w),z_2/(4-w^2),z_3/(1+w+w^2))$$
(the denominators are the Galois conjugates of $2-w$).
Then we have
\begin{align}
F_0({z'}) &= \frac 1{14}+
\sum_{\substack{a\in \OOO , \tp{a}\\ a =0\bmod p}} \exp(2\pi\ii \Tr(az'/7)
\Big( \sum_{(c)|(a)} s(c)\Big)\\
&=\frac 1{14}+
\sum_{a\in {\OOO}, \tp{a}} \exp(2\pi\ii \Tr(az/7)\label{eq:series}
\Big( \sum_{(c)|(a)} s(c)\Big).
\end{align}
Here we used 
$\sum_{(c)|(a(2-w))} s(c)= \sum_{(c)|(a)} s(c)$,
which follows since the additional summands on the left vanish.

We now use the action of ${z} \mapsto {z}+ {(1,1,1)}$ to separate
\eqref{eq:series} into eigenvectors of this action.  This leads to
$F_0,F_1,F_2,F_4$, since it is rather straightforward to see that other
possible values of $a\bmod p$ give series that are identically zero,
since in these cases $s(c) = - s(a/c)$. It is also easy to observe
that the $F_i$ above are nonzero.

The statement for $E_2$ follows from the computations of Eisenstein
series found in \cite{Freitag}.

\end{proofof}

\begin{corollary}\label{isfour}
The space of modular forms of weight $(1,1,1)$ for 
$\Gamma(p)$ has dimension $4$ and basis $\{F_0,F_1,F_2,F_4\}$. As a consequence,
the linear span of $F_i$ is a four-dimensional irreducible representation
of $\SL(2,\FF_7)$.
\end{corollary}

\begin{proof}
By Proposition \ref{zeroorfour} we know that the dimension of this 
space is either zero or four. It remains to check that the $F_i$ are 
linearly independent, but this is obvious, since $F_{i}$ is an eigenvector
of the action ${z}\mapsto {z}+(1,1,1)$ with 
eigenvalue $\zeta_7^{3i}$.
\end{proof}

\begin{remark}\label{qexp}
We can represent the modular forms for $\Gamma(p)$ as
power series in variables $q_1,q_2,q_3$ by writing $$\exp(2\pi\ii
\Tr(az)/7) = q_1^{\Tr(a (2-w)/7)}q_2^{\Tr(a(4-w^{2})/7)}
q_3^{\Tr(a(1+w+w^2)/7)}.$$  In these coordinates the expansions
are symmetric under cyclic permutation of the subscripts of the
$q_{i}$, and so in the following we put 
\[
q (a_{1},a_{2},a_{3}) = \delta \sum_{\sigma \in \ZZ /3\ZZ} \prod q_{\sigma (i)}^{a_{\sigma (i)}},
\]
where $\delta =1/3$ if $a_{1}=a_{2}=a_{3}$, and is $1$ otherwise.
With these notations, the Eisenstein series
become 
\begin{multline*}
F_{0}=1/14 +
q( 2, 2, 3)+
q( 2, 5, 7)+
q( 3, 5, 6)\\+
q( 3, 6, 5)+
q( 3, 7, 11)+
2q( 4, 10, 14)+
2q( 4, 4, 6)+\dotsb 
\end{multline*}
\begin{multline*}
F_{1} = q( 1, 1, 1)+
2q( 2, 4, 4)+
q( 2, 3, 5)+
q( 2, 6, 9)+
2q( 3, 3, 4)\\
+
2q( 3, 6, 8)+
2q( 4, 11, 16)+
2q( 4, 5, 8)+
2q( 4, 6, 7)\\
+
2q( 4, 7, 6)+
2q( 4, 9, 11)+
3q( 4, 8, 12)+\dotsb 
\end{multline*}
\begin{multline*}
F_{2} = q( 1, 2, 3)+
2q( 2, 2, 2)+
q( 2, 5, 6)+
2q( 3, 4, 6)+
2q( 3, 5, 5)\\
+
2q( 3, 7, 10)+
2q( 4, 10, 13)+
2q( 4, 12, 18)+
2q( 4, 4, 5)\\
+
2q( 4, 6, 10)+
2q( 4, 7, 9)+
2q( 4, 9, 14)+
3q( 4, 8, 8)+\dotsb 
\end{multline*}
\begin{multline*}
F_{4} = q( 1, 2, 2)+
2q( 2, 4, 6)+
2q( 3, 4, 5)+
2q( 3, 5, 4)+
2q( 3, 7, 9)\\
+
q( 3, 9, 14)+
2q( 4, 10, 12)+
2q( 4, 6, 9)+
2q( 4, 7, 8)\\
+
2q( 4, 8, 7)
+
2q( 4, 9, 13)+
3q( 4, 4, 4)+\dotsb 
\end{multline*}
\begin{multline*}
E_{2} = -1/168 +
q( 2, 2, 3)+
q( 2, 5, 7)+
8q( 3, 5, 6)+
q( 3, 6, 5)\\
+
q( 3, 1, 11 )+
9q( 4, 4, 6)+
14q( 4, 5, 5)+
14q( 4, 7, 10)\\
+
14q( 4, 8, 9)+
9q( 4, 10, 14)+ \dots 
\end{multline*}

\end{remark}

\section{Relation among $F_0,F_1,F_2,F_4,E_2$}\label{s:relation}
In this section we will find a polynomial relation among the Eisenstein
series $F_i$ and $E_2$. 
We begin by calculating the action of the group $G = \Gamma(1)\slash
\Gamma(p)\cong \SL(2,\FF_7)$ on the linear span of the $F_i$. Notice that this
group is generated by the elements $g_7$ and $g_4$ given by
\[
g_{7}\colon z \mapsto z +(1,1,1), \quad g_{4}\colon
(z_1,z_2,z_3)\mapsto \big( -\frac 1{z_1},-\frac 1{z_2},-\frac 1{z_3}\big);
\]
the element $g_{k}$ has order $k$.

First, the action of $g_{7}$ is given by $g_7 (F_i)=
\zeta_{7}^{3i}F_i$ for all $i$, which can be seen directly from
Definition \ref{defFi}.

Next, to calculate the action of $g_4$, we consider the
restrictions $\bar F_i(\tau)$ of $F_i(z)$ to the diagonal
$z=(\tau,\tau,\tau),\tau\in\HH$. These are clearly modular forms of
weight $3$ for the principal congruence subgroup $\Gamma(7)$ of
$\SL(2,\ZZ)$.  The action of $g_4$ can be recovered from the action of
$\tau\mapsto -1/\tau$ on $\bar F_i$.

First we consider $\bar F_0(\tau)$.  The calculation simplifies here,
since $\bar F_0(\tau)$ is invariant under $\tau\mapsto \tau+1$ and is
thus a modular form of weight $3$ for the larger group
$\Gamma_1(7)$. The ring of modular forms for $\Gamma_1(7)$ is easy to
describe in terms of weight one Eisenstein series \cite{toric}. We will also need
an explicit action of the Fricke involution on these weight one forms.

\begin{proposition} \emph{\cite[Theorem 4.11]{toric}}
Let $q=\exp(2\pi\ii\tau)$. For $a=1,2,3$ let $s_{a}$ be the series 
$$s_a(\tau)=\frac 12-\frac a7+\sum_{n>0}q^n\sum_{d\vert
n}(\delta_{d}^{a\bmod7}-\delta_{d}^{-a\bmod7}),$$ where
$\delta^{\alpha}_{\beta}$ takes the value $1$ if $\alpha$ is the
residue $\beta$, and is $0$ otherwise.  Then $\{s_{1},s_{2},s_{3} \}$
forms a basis of the space of modular forms of weight one for
$\Gamma_1(7)$.
\end{proposition}

\begin{remark} Explicitly,
$$
\begin{array}{l}
s_1=1/2-1/7+q+q^2+q^3+q^4+q^5+q^7+2q^8+q^9+q^{10}+q^{11}+q^{14}+2q^{15}+O(q^{16}),\\
s_2=1/2-2/7+q^2+q^4-q^5+q^6+q^8+q^9+q^{14}-q^{15}+O(q^{16}),\\ 
s_3=1/2-3/7+q^3-q^4+q^6-q^8+q^9+q^{10}-q^{11}+q^{15}+O(q^{16}).
\end{array}
$$
In addition, these $s_i$ satisfy the equation $7
(s_1^2+s_2^2+s_3^2)=5(s_1+s_2-s_3)^2$, which establishes an
isomorphism of $X_1(7)$ with a conic in $\PP^2 (\CC)$.  For more
details see \cite{modular}.
\end{remark}

\begin{proposition}\label{restr0}
The restriction $\bar F_0(\tau)$ of $F_0({z})$ is given by
\begin{equation}\label{eq:f0}
\bar F_0=7(s_1+s_2-s_3)^3-147s_1s_2s_3.
\end{equation}

\end{proposition}

\begin{proof}
The line bundle of weight three forms on $X_1(7)$ has degree $6$.  One
can compute that the difference of the left and right sides of
\eqref{eq:f0} vanishes at $\ii\infty$ to order more than $6$, which
means the difference vanishes.
\end{proof}

We now calculate the Fricke involution of $\bar F_0(\tau)$.  We will
only require a few terms in the resulting $q$-expansion.
\begin{proposition}\label{frickebar}
We have
\begin{multline}\label{eq:qexpoffrickebar}
\frac 1{\tau^3} \bar F_0(-\frac 1{7\tau}) = \\ 
49\ii \sqrt{7}
(\frac 1{14}+q^3+3q^5+5q^6+3q^7+15q^{10}+21q^{12}+21q^{13}+15q^{14} + O(q^{15})).
\end{multline}
\end{proposition}

\begin{proof}
We use 
$$
\frac 1{\tau} s_i(-\frac 1{7\tau}) = C + \sum_{n>0}q^n\sum_{d\vert n}(\zeta_7^{di}-\zeta_7^{-di}),
$$
where $C$ is an irrelevant constant,
which follows from the calculations of \cite{vanish}.
This shows 
$$
\begin{array}{l}
-\frac 1{\tau} s_1(-\frac 1{7\tau}) = (\zeta_7-\zeta_7^{-1})s_1+ (\zeta_7^2-\zeta_7^{-2})s_2+ (\zeta_7^3-\zeta_7^{-3})s_3\\
-\frac 1{\tau} s_2(-\frac 1{7\tau}) = (\zeta_7^2-\zeta_7^{-2})s_1+ (\zeta_7^4-\zeta_7^{-4})s_2+ (\zeta_7^6-\zeta_7^{-6})s_3\\
-\frac 1{\tau} s_3(-\frac 1{7\tau}) = (\zeta_7^3-\zeta_7^{-3})s_1+ (\zeta_7^6-\zeta_7^{-6})s_2+ (\zeta_7^9-\zeta_7^{-9})s_3
\end{array}
$$
Together with Proposition \ref{restr0} this yields \eqref{eq:qexpoffrickebar}.
\end{proof}

\begin{proposition}\label{g4f0}
The action of $g_4$ on $F_0$ is given by 
$$
g_4 ( F_0) = -\frac {\ii}{\sqrt 7}(F_0+F_1+F_2+F_4).
$$
\end{proposition}

\begin{proof}
By definition of $g_4$,
$$g_4F_0(z_1,z_2,z_3)=(-\frac 1{z_1z_2z_3})F_0(-\frac 1{z_1},-\frac 1{z_2},-\frac 1{z_3}).$$
We apply Proposition \ref{frickebar} with $\tau$ replaced by $\tau/ 7$ 
to get
$$g_4F_0(\tau,\tau,\tau) = -\frac 1{\tau^3} \bar F_0(-\frac 1\tau) = 
-\frac {\ii}{\sqrt 7} (\frac 1{14}+q^{\frac 37}+3q^{\frac 57}+5q^{\frac 67}+3q+O(q^{\frac 87})).
$$
Since $g_4F_0(z_1,z_2,z_3)$ is a linear combination
of $F_0,F_1,F_2,F_4$, it remains to compare this expansion of it with 
those of the restrictions $\bar F_i$.
\end{proof}

\begin{proposition}\label{theaction}
The action of $g_4$ and $g_7$ in the basis $\{F_0,F_1,F_2,F_4\}$ is given
by the matrices
$$
\gamma_4=-\frac {\ii}{\sqrt 7}
\left(
\begin{array}{cccc}
1&2&2&2\\
1&a_1&a_3&a_2\\
1&a_3&a_2&a_1\\
1&a_2&a_1&a_3
\end{array}
\right)
\hskip .1in
\gamma_7=
\left(
\begin{array}{cccc}
1&0&0&0\\
0&\zeta_7^3&0&0\\
0&0&\zeta_7^6&0\\
0&0&0&\zeta_7^5
\end{array}
\right)
$$
where $a_1=\zeta_7+\zeta_7^6$, 
$a_2=\zeta_7^2+\zeta_7^5$, $a_3=\zeta_7^3+\zeta_7^4$. 
\end{proposition}

\begin{proof}
A computer computation shows that the above matrices $\gamma_4$ and
$\gamma_7$ define a representation of $G$. It is well known that the
group $G$ has a unique (up to isomorphism) $4$-dimensional complex
representation such that $g_7$ has eigenvalues
$1,\zeta_7^3,\zeta_7^5,\zeta_7^6$.  Hence we can find a matrix $T\in
\GL(4,\CC )$ such that $\gamma_4$ and $\gamma_7$ are $T$-conjugates of
the matrices of $g_4$ and $g_7$ in the basis $F_0,F_1,F_2,F_4$.

Since the action of $g_7$ is known to be given by the matrix
$\gamma_7$, we get $T\gamma_7=\gamma_7T$, which implies that $T$ is
diagonal.  Write $T=\diag(t_0,t_1,t_2,t_4)$.  We have
$$g_4(F_0)=T\gamma_4 T^{-1}(F_0)=
 -\frac {\ii}{\sqrt 7} t_0^{-1}(t_0F_0+t_1F_1+t_2F_2+t_4F_4).$$
Together with Proposition \ref{g4f0}, this shows that $T$ is 
a multiple of the identity, which finishes the proof.
\end{proof}

\begin{proposition}\label{ugly}
The Eisenstein series $E_2$, $F_0$, $F_1$, $F_2$ and $F_4$ satisfy
the polynomial relation $P_8(F_0,F_1,F_2,F_4,E_2)=0$, where 
$$
{
\begin{array}{rl}
P_8=&\bigl(\frac 67 E_2\bigr)^4 
- 3\bigl(\frac 67 E_2\bigr)^2\bigl(2F_0^4 + 6F_0F_1F_2F_4 + (F_2F_4^3 + F_1^3F_4 + F_1F_2^3)\bigr)
\\& + \bigl(\frac 67 E_2\bigr)
\bigl(-8F_0^6 + 20F_0^3F_1F_2F_4 + 10F_0^2(F_2F_4^3 + F_1^3F_4 + F_1F_2^3) 
\\&
+ 10F_0(F_1^2F_4^3 + F_2^3F_4^2 + F_1^3F_2^2) + (15F_1^2F_2^2F_4^2 + F_1F_4^5 + F_2^5F_4 + F_1^5F_2)\bigr)
\\&
-3F_0^8 + 38F_0^5F_1F_2F_4 + 13F_0^4(F_2F_4^3 + F_1^3F_4 + F_1F_2^3) 
\\&
-46F_0^3(F_1^2F_4^3 + F_2^3F_4^2 + F_1^3F_2^2) 
+ F_0^2(5F_1F_4^5 + 5F_2^5F_4 + 5F_1^5F_2 + 23F_1^2F_2^2F_4^2 ) 
\\&
-2F_0 (F_4^7 + F_2^7 + F_1^7 + 4F_1F_2^2F_4^4 + 4F_1^4F_2F_4^2 + 4F_1^2F_2^4F_4) 
\\&
+ (2F_2^2F_4^6 + 2F_1^6F_4^2 + 2F_1^2F_2^6 - 5F_1^3F_2F_4^4 - 5F_1F_2^4F_4^3 - 5F_1^4F_2^3F_4).
\end{array}
}
$$
\end{proposition}

\begin{proof}
The dimensions of the invariants of the $k$-th symmetric powers of the
$4$-dimensional representation of $G$ for small $k$ are given by the
following table:

\begin{center}
\medskip
\begin{tabular}{|c|ccccc|}
\hline
$k$&$0$&$2$&$4$&$6$&$8$\\\hline
$\dim(\Sym^k(V_4)^G)$& $1$ & $0$&$1$&$1$&$3$\\
\hline
\end{tabular}
\medskip
\end{center}

As a result, there is a $6$-dimensional space of polynomials of total
degree $8$ in weight one variables $F_i$ and weight two variable $E_2$
that are invariant under the action of $G$. Each such polynomial gives
a weight $(8,8,8)$ modular form with respect to the group $\Gamma(1)$.
We claim the space of $(8,8,8)$ modular forms for $\Gamma(1)$ has
dimension 5.

To see this, we use the trace formula 
\eqref{eqn:tracefmla}.  This says that if $k\geq 2$ is even, the
dimension of the space of cusp forms on $\Gamma (1)$ is
\[
\vol (\SL (2,\OOO)\backslash \HH^{3}) (k-1)^{3} + e +c,
\]
where as before $e$ (respectively $c$) represents the contribution of
the elliptic points (resp., cusps).  We have $c=0$ as before, but now
we have nontrivial elliptic points and must evaluate $e$.  It is
well-known that the elliptic elements have orders $2,3,7$, and it is
not hard to find representatives of them modulo $\SL (2,\OOO)$: there
are four points of each order, giving $12$ in all.  (The elements of
orders $2$ and $3$ come from conjugates of the usual elliptic elements
of $\SL (2,\ZZ)$, and the elements of order $7$ come from writing the
ring of integers of $\QQ (\zeta_{7})$ as a free rank 2 $\OOO$-module
and looking at the action of $\zeta_{7}$.) Each elliptic point $P$ has
a rational triple $\alpha = (\alpha_{1}, \alpha_{2}, \alpha_{3})$ of
``rotation numbers'' attached to it that reflects the action of the
stabilizer of $P$.  The contribution of $P$ to $e$ is then
\[
\frac{1}{N (P)} \sum_{j=1}^{N (P)} \prod_{\ell =1}^{3} \frac{\exp
(2\pi\ii kj\alpha_{\ell } )}{1-\exp (2\pi \ii j\alpha_{\ell})},
\]
where $N (P)$ is the order of $P$.  The rotation numbers for our 12
points are shown in Table \ref{rotationtable}.  From this it is easy
to compute that the dimension of the cuspidal subspace of parallel
weight $8$ is $4$.  There is one Eisenstein series of parallel
weight $8$ since the class number of $\OOO $ is one, and hence the
dimension of the full space of parallel weight $8$ forms is $5$. 

Since the space of $(8,8,8)$ modular forms has dimension $5$, the
invariant polynomials in $F_i,E_2$ must satisfy at least one linear
relation, and it then becomes a matter of finding it.  We did this
using Magma \cite{magma} to find a six-element spanning set of
invariant degree $8$ polynomials in $F_i,E_{2}$.  Looking at their
Fourier expansions at the cusp $(\ii\infty)^3$, we observed that their
span is of dimension at least five, and the only linear relation can
be the multiple of the one in the statement.  We leave the
computational details to the reader.
\end{proof}

\begin{table}
\begin{center}
\begin{tabular}{|c|c||c|}
\hline
Order&Multiplicity&Rotation numbers\\
\hline
$2$&4&$(1/2,1/2,1/2)$\\
$3$&1&$(1/3,1/3,1/3)$\\
$3$&3&$(1/3, 1/3, 2/3)$\\
$7$&1&$(1/7, 2/7, 4/7)$\\
$7$&3&$(1/7, 2/7, 3/7)$\\
\hline
\end{tabular}
\end{center}
\medskip
\caption{Rotation numbers of elliptic points\label{rotationtable}}
\end{table}

\section{The symmetric Hilbert modular threefold}\label{s:symmetric}
Recall that $K/\QQ$ is a Galois extension with Galois group $\Gal
(K/\QQ)$ isomorphic
to $\ZZ /3\ZZ$.  The action of $\Gal (K/\QQ)$ on the real places of
$K$ induces an action of $\ZZ /3\ZZ$ on $\HH^{3}$ by cyclically
permuting the coordinates.  This extends to an action on the Hilbert
modular threefold $X$.

\begin{definition}\label{def:symmhm3fd}
The \emph{symmetric Hilbert modular threefold} $X_{\Gal}$ is the
quotient of $X$ by the induced action of $\ZZ /3\ZZ$.
\end{definition}

The goal of this section is to explicitly describe $X_{\Gal}$.

\begin{proposition}\label{justugly}
The polynomial $P_8$ from Proposition \ref{ugly} generates the ideal of 
relations on $F_i,E_2$.
\end{proposition}

\begin{proof}
We first observe that $P_8$ is irreducible.  Indeed, any factors would
be acted upon by the group $G$. Since some of the degrees of the
factors in $E_2$ are four, and $G$ has no nontrivial permutation
representation on four or fewer elements, and since it has no
non-trivial one-dimensional representations, each factor would have to
be $G$-invariant. There are no invariant polynomials of degree two in
$F_i$, so the only possibility is
$P_8=(E_2^2+aP_4(F_i))(E_2^2+bP_4(F_i))$ where $P_4$ is the generator
of the one-dimensional invariant space $\Sym^4(V_4)^G$. But this is
easily seen to be impossible by looking at the coefficient by $E_2$.

Then it remains to show that the Eisenstein series $F_i$ are
algebraically independent, as any other relation on $(F_i,E_2)$
together with $P_8$ would lead to a relation among the $F_i$.  In
order to see this, we observe the transformation properties of the
$F_{i}$ imply that any polynomial relation among the $F_i$ must be
homogeneous. The existence of such a relation would in turn imply a
polynomial relation among $T_1={F_1}/{F_0}$, $T_2={F_2}/{F_0}$ and
$T_3= {F_4}/{F_0}$.  Consequently, the Jacobian
$\det(\partial_{i}T_j)$ would have to be zero, where $\partial_{i} =
\partial /\partial_{q_{i}}$.  But computing with Fourier expansions
shows that up to a constant, this Jacobian equals
\begin{multline*}
-2744q_2^5q_1^4q_0^2 + (5488q_2^{13}q_1^9 - 2744q_2^8q_1^7 -
2744q_2^9q_1^6 \\
+ 5488q_2^3q_1^5 - 2744q_2^4q_1^4 +
5488q_2^5q_1^3)q_0^3 + O (q_{0}^{4}),
\end{multline*}

and is therefore nonzero for most points in the neighborhood
of the cusp.
\end{proof}

We will need some 
lemmas that describe the fixed loci of the Galois action.
Let $\sigma:X_{sm}\to X_{sm}$ be the generator of the 
Galois group. It is the extension of the cyclic permutation of the coordinates of $\HH^3$.
\begin{lemma}\label{fixedloci}
The fixed loci of $\sigma$ on $X_{sm}$ are of dimension at most one.
\end{lemma}

\begin{proof}
We first consider the fixed points in the finite part.
Suppose that an image of $(z_1,z_2,z_3)\in \HH^3$ is fixed by $\sigma$ in $X$.
This means that there exists $A\in \Gamma(p)$ such that
$(A,A^{\sigma},A^{\sigma^2})(z_1,z_2,z_3)=(z_3,z_1,z_2)$.
This leads to $AA^{\sigma}A^{\sigma^2}z_3=z_3$. The Galois conjugate of
$B=AA^{\sigma}A^{\sigma^2}$ is $B^\sigma=A^{\sigma}A^{\sigma^2}A
=A^{-1}BA$, so ${\rm Tr}(B)$ is rational. On the other hand, ${\rm Tr}(B)$
is in $\OOO$ and is equal to $2\mod p$.
This shows that ${\rm Tr}(B)$ is an integer which is equal to $2\mod 7$.
Since $z_3\in \HH$, the only possibility is $B={\rm Id}$. For each
$A$ with this property we get a curve in $X_{sm}$.

As far as the boundary divisors go, we only need to show that $\sigma$
does not fix $E_i$ or $D_i$ pointwise. Indeed, $\sigma$ permutes
$D_i$, and hence it permutes their intersections with $E_i$. Thus
$\sigma$ acts nontrivially on each $E_i$.
\end{proof}

\begin{remark}
We believe that the only fixed points of $\sigma$ in the finite part of $X_{sm}$
are the image of the diagonal in $\HH^3$, but we will not need this fact.
\end{remark}

Consider the quotient map by the Galois group $X_{sm}\to Y$.
\begin{lemma}\label{trivchar}
Consider the natural linearization of the line bundle $L$ given by
the pullbacks of holomorphic functions on $\HH^3$. Then for every $\sigma$-fixed
point $x$ of $X_{sm}$, the corresponding character of the Galois group in the fiber
$L_x$ is trivial.
\end{lemma}

\begin{proof}
The fixed locus of any finite order automorphism of a smooth variety
is a disjoint union of smooth subvarieties, and the character of the
restriction of a line bundle is locally constant.

Over cusps, the linearized line bundle $L$ is isomorphic to $\mathcal
O$, so the statement is clear. The proof of Lemma \ref{fixedloci}
shows that every $\sigma$-fixed curve in the finite part contains a
point at infinity in its closure in $X_{sm}$. Hence, the character of
$\sigma$ on $L_x$ is trivial on each of these components as well.
\end{proof}

\begin{theorem}\label{symhm3fd}
The symmetric Hilbert modular threefold $X_{\Gal}$ is isomorphic to
the degree $8$ hypersurface in $\PP(1,1,1,1,2)$ given by $P_8=0$.
\end{theorem}

\begin{proof}
We will prove that the ring $M=M^{\rm Gal}$ of modular forms of
parallel weight invariant under the Galois action is isomorphic to
the subring $M'$ generated by $F_i$ and $E_2$, which clearly
suffices in view of Proposition \ref{justugly}.

First of all $M'$ is a subring of $M$. Indeed, $F_i$ and $E_2$
are invariant under the permutation of variables $(z_1,z_2,z_3)
\to (z_2,z_3,z_1)$, since the Galois group acts trivially on
${\OOO}/p$. By Proposition \ref{justugly} the Poincare series of
$M'$ is $f(t)=(1+t^2+t^4+t^6)(1-t)^{-4}$.
so it suffices to show that the Poincare series of $M$ is
the same. For degree $k\leq 3$ all the modular forms of weight
$(k,k,k)$ (Galois invariant or not) lie in $M'$, since the
dimension of the degree $k$ part of $M'$ is equal to $h^0(kL)$.
For $k\geq 4$ the dimension of $M'_{\deg=k}$ is
$$
\frac 23 k^3-2k^3+\frac {34}3 k -10
$$
so all one needs to do is to calculate the dimension of the Galois-invariant
part of the space of $(k,k,k)$ forms.

The Eisenstein series of weight $(k,k,k)$ are Galois-invariant,
because the Galois group fixes the cusps. Thus the dimension of the
degree $k$ part of $M$ is $8$ plus the dimension of the Galois-invariant
part of $H^0(X_{sm},kL-D-E)$. The line bundle
${\mathcal O}_{X_{sm}}(kL-D-E)$ for $k\geq 3$ satisfies Kawamata vanishing,
so it has no higher cohomology and we can use
the holomorphic Lefschetz formula
for the trace of the action of the Galois group on the cohomology of $H^i(L)$.
Since the Euler characteristics of ${\mathcal O}_{X_{sm}}(kL-D-E)$
is $2(k-1)^3$, the Euler characteristics of the Galois-invariant part
is $\frac 23(k-1)^3+h(k)=
\frac 23(k-1)^3+2{\rm Re}({\rm Trace}(\sigma))$ where  $\sigma$ is the
operator on the cohomology of the linearized line bundle
${\mathcal O}_{X_{sm}}(kL-D-E)$ that corresponds to the generator of the
Galois group.

We will now use the Lemmas \ref{fixedloci} and \ref{trivchar} to show
that $h(k)$ is a linear function of $k$. Indeed, the holomorphic
Lefschetz formula gives a sum over the components $Z_i$ of the fixed
locus of $\sigma$ of integrals of some fixed class times the
equivariant Chern character of $kL$ restricted to $Z_i$. This will
give a mod 3 quasi-polynomial by Lemma \ref{fixedloci}. Moreover, it
will in fact be polynomial, since there will be no factors of the form
$\exp((2\pi \ii/3) k)$ by Lemma \ref{trivchar}.


Thus it suffices to find $h(k)$ for two values of $k$ in order to calculate
the graded dimension of $M$. For $k=2$ we are dealing with the
canonical class $2L-D-E$ in its natural linearization.
The only cohomology occurs at $H^0$ and $H^3$.
The cohomology for $H^0$ is Galois-invariant, since it can be written
in terms of the elements of $M'$. The cohomology at $H^3$ is Serre dual
to $H^0(X_{sm},{\mathcal O})$, and the natural linearizations of $K$
and $\mathcal O$ are compatible
with Serre duality.  Thus we have $h(2)=(3-1)-\frac 23(2-1)^3 = \frac 43$.
For $k=3$ we have ${\mathcal O}(3L-D-E)$. Since all weight $(3,3,3)$ forms
are Galois invariant, so are all cusp forms of this weight. There are no
higher cohomology, and we have $\chi(3L-D-E)=2(3-1)^3=16$.   
Hence we have
$$
h(3)=16 - \frac 23(3-1)^3= \frac {32}3.
$$
This leads to $h(k)=\frac 13(28k-52)$, which then shows that
for $k\geq 4$ $\dim M_k = \dim M'_k$.
\end{proof}

\section{The canonical model}\label{s:canonical}
The goal of this section is to prove that the partial resolution
$X_{ch}$ is the canonical model of the field of modular functions
for $\Gamma(p)$.

It is well-known that global holomorphic top forms on $X_{sm}$ can be
identified with cusp forms of parallel weight $(2,2,2)$.  In fact, we
can calculate the basis of these forms in our situation.
\begin{proposition}\label{basis}
The dimension of the space $H^0(X_{sm},K_{sm})$ is $3$. Its basis is
given by $2F_0F_1-F_4^2,2F_0F_2-F_1^2,2F_0F_4-F_2^2$.
\end{proposition}

\begin{proof}
The dimension statement follows from the trace formula, see
\cite{Freitag}. To find the basis, observe that the space of modular
forms of weight $(2,2,2)$ is a $G$-module of dimension $11$. It is
easy to show that $E_{2}$ and the degree two monomials in $F_0,F_{1},
F_{2}, F_4$ are linearly independent and hence give a basis of the
weight $(2,2,2)$ forms. As a $G$-module, the space of weight $(2,2,2)$
forms decomposes into the direct sum of irreducibles of degrees $1, 3,
7$. This implies that the cusp forms of parallel weight $2$ correspond
to the three-dimensional summand of the space of quadrics in
$F_0,\ldots,F_4$. It is easy to check that the three quadrics in the
statement generate this subspace. Alternatively, one can look at the
eight translates of the cusp $(\ii\infty)^3$ given by
$(F_0:F_1:F_2:F_4)=(1:0:0:0)$ by the action of $\SL(2,\FF _7)$ and
look at the quadrics that vanish at these points.
\end{proof}

\begin{proposition}\label{2K}
Identify the space of sections of $2K$ on $X_{ch}$ (or on any other
model of its function field with canonical singularities) with the
subspace of $H^0(X_{ch},2L)$ that vanishes twice at the exceptional
divisor $D$.  Under this identification, the following eight elements
of $H^0(X_{ch}, 2L)$ give sections of $2K$:
$$
\begin{array}{l}
F_0^4-
(\frac 67E_2)^2+3F_0F_1F_2F_4+\frac 12
(F_2F_4^3+F_4F_1^3+F_1F_2^3),\\
 4F_0^4 +\frac {24}7E_2F_0^2 - 10F_4F_2F_1F_0 + F_4F_1^3 + F_2^3F_1 + F_4^3F_2,\\
4F_1F_0(F_0^2+\frac 67E_2) - 10F_4^2F_0^2 - 4F_4F_2^2F_0 -4F_4F_2F_1^2 + F_2^4 +\frac {12}7E_2F_4^2,\\
 4F_2F_0(F_0^2+\frac 67E_2) - 10F_1^2F_0^2 - 4F_1F_4^2F_0 -4F_1F_4F_2^2 + F_4^4 +\frac {12}7E_2F_1^2,\\
 4F_4F_0(F_0^2+\frac 67E_2) - 10F_2^2F_0^2 - 4F_2F_1^2F_0 -4F_2F_1F_4^2 + F_1^4 +\frac {12}7E_2F_2^2,\\
2F_2F_1F_0^2 + 2F_1^3F_0 - 2F_4^2F_2F_0  + F_4^2F_1^2 -\frac {12}7E_2F_2F_1 + 2F_4F_2^3,\\
2F_1F_4F_0^2 +2F_4^3F_0 - 2F_2^2F_1F_0  +F_2^2F_4^2 -\frac {12}7E_2F_1F_4 +2F_2F_1^3,\\
2F_4F_2F_0^2 + 2F_2^3F_0 - 2F_1^2F_4F_0  + F_1^2F_2^2 -\frac {12}7E_2F_4F_2 +2F_1F_4^3.
\end{array}
$$
\end{proposition}

\begin{proof}
It is straightforward to see that these forms vanish twice at the part
of $D$ that sits over the image of $(\ii\infty)^3$ in view of the
explicit expansions in Remark \ref{qexp}.  It suffices to check that
the span of the above series is invariant under the action of $G$
given by Proposition \ref{theaction}.  We have done the latter
calculation by GP-Pari. The first element is invariant under $G$,
while the latter seven span a seven-dimensional irreducible
representation of $G$.
\end{proof}

\begin{lemma}\label{ampleonD1}
The $\QQ$-Cartier divisor $K_{ch}$ is ample on the divisor $D_1$ on
$X_{ch}$.
\end{lemma}

\begin{proof}
The fan of the toric surface $D_1$ on $X_{ch}$ is given on the right
of Figure \ref{fig5}, with $K_{ch}\vert D_1$ given by the
piecewise-linear function whose values on the generators of
one-dimensional cones of the fan are given in the Figure.  Since the
cone of effective curves on a toric surface is a finite polyhedral
cone spanned by the infinity divisors, it suffices to check that the
intersection number $K_{ch}M$ is positive for each of the six infinity
divisors $M$. A standard toric geometry calculation shows that for
each infinity divisor $M$ we have $K_{ch}M=1/2$.  If one would rather
avoid the singularities and the fractional intersection numbers, the
calculation can be done on $D_1\subset X_{sm}$, in view of the
projection formula.
\end{proof}

\begin{theorem}\label{canmodel}
The partial resolution $X_{ch}$ is the canonical model of the field of
modular functions of $\Gamma(p)$.
\end{theorem}

\begin{proof}
Recall that every field of functions of general type has a unique 
canonical model which is characterized in its birational class
by the properties of having canonical singularities and an ample 
canonical line bundle.

Clearly, $X_{ch}$ has canonical singularities, so it suffices to show
that $K_{X_{ch}}$ is ample. Note that $K_{X_{ch}}$ is only a Weil
divisor, but $2K_{X_{ch}}$ is Cartier. We will use the Nakai-Moishezon
criterion, see \cite[Theorem 1.37]{KollarMori}. By Proposition
\ref{intnumbers} we have $K_{X_{ch}}^3>0$, since the divisor $K-\frac
12E$ is the pullback of $K_{X_{ch}}$.  Let $C$ be an irreducible curve
on $X_{ch}$.  If $C$ lies in an infinity divisor, then $K_{X_{ch}}C>0$
by Lemma \ref{ampleonD1}.  So we can consider the case where the
generic point of $C$ lies in the finite part of $X_{ch}$.

Denote by $R_j, j=1,\ldots,11$ the polynomials in $F_i$ and $E_2$ from
Propositions \ref{basis} and \ref{2K}.  If $K_{X_{ch}}C\leq 0$ then
$C$ maps to a point in the weighted projective space
$\PP(1,1,1,2,2,2,2,2,2,2,2)$ given by the $R_j$.  Consider the
$5\times 11$ matrix of the partial derivatives of $R_j$ with respect
to $F_i$ and $E_2$.  For every point in the image of the finite part
of $C$ in $X_{\rm Gal}$ the rank of this matrix is less than five. We
have calculated in Magma \cite{magma} the Hilbert series of the scheme
defined by the size five minors of this matrix, to show that this
scheme is zero-dimensional.  This shows that $K_{ch}C>0$.

It remains to show that $K_{X_{ch}}^2S>0$ for any surface 
$S\subset X_{ch}$. It is easy to see that the quadrics from
Proposition \ref{basis} have no common components, so 
$K_{X_{ch}}S$ is an effective curve except perhaps if $S$ is a 
divisor at infinity. Then Lemma \ref{ampleonD1} finishes the 
argument.
\end{proof}

\section{An octic with $84$ singularities of type $A_2$}\label{s:octic}
In this section we describe explicitly an octic $W\subseteq \PP^3$
with $84$ isolated singular points of type $A_2$.

Consider the intersection of the Hilbert modular threefold $X_{Gal}$ with
the hypersurface $E_2=0$. The resulting surface $W$ is given by the
degree $8$ equation $Q (F_{0}, F_{1}, F_{2}, F_{4}) = 0$, where
\begin{multline*}
Q = -3F_0^8 + 38F_0^5F_1F_2F_4 + 13F_0^4(F_2F_4^3 + F_1^3F_4 + F_1F_2^3) 
\\
-46F_0^3(F_1^2F_4^3 + F_2^3F_4^2 + F_1^3F_2^2) 
+ F_0^2(5F_1F_4^5 + 5F_2^5F_4 + 5F_1^5F_2 + 23F_1^2F_2^2F_4^2 ) 
\\
-2F_0 (F_4^7 + F_2^7 + F_1^7 + 4F_1F_2^2F_4^4 + 4F_1^4F_2F_4^2 + 4F_1^2F_2^4F_4) 
\\
+ (2F_2^2F_4^6 + 2F_1^6F_4^2 + 2F_1^2F_2^6 - 5F_1^3F_2F_4^4 - 5F_1F_2^4F_4^3 - 5F_1^4F_2^3F_4)
\end{multline*}

in homogeneous coordinates $F_0,F_1,F_2,F_4$.
\begin{proposition}\label{prop:octic}
The octic $W$ in $\PP^3$ has $84$ singular points of type $A_2$.
\end{proposition}

\begin{proof}
Note that $X_{Gal}$ has singularities of type $A_2$ along the image of
the diagonal in $\HH^3$. The Hilbert modular form $E_2$ restricts to
the $SL_{2} (\ZZ)$-Eisenstein series $g_6$ on the diagonal. In fact,
since $g_6$ has simple zeroes, at each such zero the surface $E_2=0$
in $\HH^3$ is nonsingular. The gradient of $E_2$ is Galois-invariant,
so the tangent space of $E_2=0$ has eigenvalues $(-\frac
12\pm\frac{\sqrt 3}2)$. This shows that shows that $E_2=0$ is
transversal to the diagonal. Thus, $W$ is singular at the images of
the zeroes of $g_6$, with singularities of type $A_2$, and perhaps at
some other points or curves.  There are $84$ such zeroes of $g_6$
modulo the group action, which form the orbit of $(\ii,\ii,\ii)$.

A Magma calculation shows that the Hilbert polynomial of the radical of
the Jacobian ideal of the equation of $W$ is $84$.  Consequently, $W$ has
exactly $84$ isolated singular points, which are the zeroes of $g_6$
as above.
\end{proof}

\begin{remark}
We believe that the $84$ singular points of type $A_2$ is the best
lower bound for octic to date. The upper bound of $98$ can be
established by the work of Miyaoka \cite{miyaoka}. See also
\cite{labs} for detailed discussion about related problems. It would
be interesting to see if our octic is in any way related to the
Endra\ss\ octics that have 168 singular points of type $A_1$.
\end{remark}

\bibliographystyle{amsplain_initials}
\bibliography{paper}

\providecommand{\bysame}{\leavevmode\hbox to3em{\hrulefill}\thinspace}
\providecommand{\MR}{\relax\ifhmode\unskip\space\fi MR }
\providecommand{\MRhref}[2]{%
  \href{http://www.ams.org/mathscinet-getitem?mr=#1}{#2}
}
\providecommand{\href}[2]{#2}
\begin{thebibliography}{10}

\bibitem{vanish}
L.~A. Borisov and P.~E. Gunnells, \emph{Toric modular forms and nonvanishing of
  {$L$}-functions}, J. Reine Angew. Math. \textbf{539} (2001), 149--165.

\bibitem{toric}
L.~A. Borisov and P.~E. Gunnells, \emph{Toric varieties and modular forms},
  Invent. Math. \textbf{144} (2001), no.~2, 297--325.

\bibitem{modular}
L.~A. Borisov, P.~E. Gunnells, and S.~Popescu, \emph{Elliptic functions and
  equations of modular curves}, Math. Ann. \textbf{321} (2001), no.~3,
  553--568.

\bibitem{magma}
W.~Bosma, J.~Cannon, and C.~Playoust, \emph{The {M}agma algebra system. {I}.
  {T}he user language}, J. Symbolic Comput. \textbf{24} (1997), no.~3-4,
  235--265, Computational algebra and number theory (London, 1993).

\bibitem{ehlers}
F.~Ehlers, \emph{Eine {K}lasse komplexer {M}annigfaltigkeiten und die
  {A}ufl\"osung einiger isolierter {S}ingularit\"aten}, Math. Ann. \textbf{218}
  (1975), no.~2, 127--156.

\bibitem{Freitag}
E.~Freitag, \emph{Hilbert modular forms}, Springer-Verlag, Berlin, 1990.

\bibitem{frobenius}
G.~Frobenius, \emph{{\"Uber Gruppencharaktere}}, {Sitzungsberichte der
  K\"onigliche Preu\ss ichen Akademie der Wissenschaften zu Berlin} (1896),
  985--1021, {(Gesammelte Abhandlungen v. III, pp. 1--37)}.

\bibitem{Grundman}
H.~G. Grundman, \emph{Explicit resolutions of cubic cusp singularities}, Math.
  Comp. \textbf{69} (2000), no.~230, 815--825.

\bibitem{dedekind}
P.~E. Gunnells and R.~Sczech, \emph{Evaluation of {D}edekind sums, {E}isenstein
  cocycles, and special values of {$L$}-functions}, Duke Math. J. \textbf{118}
  (2003), no.~2, 229--260.

\bibitem{hecke}
E.~Hecke, \emph{{Analytische Funktionen und algebraische Zahlen, zweiter
  Teil}}, Abh. Math. Sem. Hamburg Univ. \textbf{3} (1924), 213--236.

\bibitem{hirz}
F.~Hirzebruch, \emph{Hilbert modular surfaces}, Enseignement Math. (2)
  \textbf{19} (1973), 183--281.

\bibitem{hirz.5}
F.~Hirzebruch, \emph{The {H}ilbert modular group for the field {${\mathbb
  Q}(\sqrt 5)$}, and the cubic diagonal surface of {C}lebsch and {K}lein},
  Uspehi Mat. Nauk \textbf{31} (1976), no.~5(191), 153--166, Translated from
  the German by Ju. I. Manin.

\bibitem{hirz.bonn}
F.~Hirzebruch, \emph{The ring of {H}ilbert modular forms for real quadratic
  fields in small discriminant}, Modular functions of one variable, {VI}
  ({P}roc. {S}econd {I}nternat. {C}onf., {U}niv. {B}onn, {B}onn, 1976),
  Springer, Berlin, 1977, pp.~287--323. Lecture Notes in Math., Vol. 627.

\bibitem{h.vdv}
F.~Hirzebruch and A.~Van~de Ven, \emph{Hilbert modular surfaces and the
  classification of algebraic surfaces}, Invent. Math. \textbf{23} (1974),
  1--29.

\bibitem{hunt.nice}
B.~Hunt, \emph{Nice modular varieties}, Experiment. Math. \textbf{9} (2000),
  no.~4, 613--622.

\bibitem{janus}
B.~Hunt and S.~H. Weintraub, \emph{Janus-like algebraic varieties}, J.
  Differential Geom. \textbf{39} (1994), no.~3, 509--557.

\bibitem{kawa}
Y.~Kawamata, \emph{A generalization of {K}odaira-{R}amanujam's vanishing
  theorem}, Math. Ann. \textbf{261} (1982), no.~1, 43--46.

\bibitem{KollarMori}
J.~Koll{\'a}r and S.~Mori, \emph{Birational geometry of algebraic varieties},
  Cambridge Tracts in Mathematics, vol. 134, Cambridge University Press,
  Cambridge, 1998, With the collaboration of C. H. Clemens and A. Corti,
  Translated from the 1998 Japanese original.

\bibitem{labs}
O.~Labs, \emph{Dessins d'enfants and hypersurfaces with many
  {$A_j$}-singularities}, J. London Math. Soc. (2) \textbf{74} (2006), no.~3,
  607--622.

\bibitem{miyaoka}
Y.~Miyaoka, \emph{The maximal number of quotient singularities on surfaces with
  given numerical invariants}, Math. Ann. \textbf{268} (1984), no.~2, 159--171.

\bibitem{schur}
I.~Schur, \emph{{Untersuchungen \"uber die Darstellungen der endlichen Gruppen
  durch gebrochene lineare Substitutionen}}, J. Reine Angew. Math. \textbf{132}
  (1907), 85--137.

\bibitem{shim}
H.~Shimizu, \emph{On discontinuous groups acting on a product of upper half
  planes}, Ann. of Math. \textbf{77} (1963), 33--71.

\bibitem{shimura}
G.~Shimura, \emph{Introduction to the arithmetic theory of automorphic
  functions}, Publications of the Mathematical Society of Japan, vol.~11,
  Princeton University Press, Princeton, NJ, 1994, Reprint of the 1971
  original, Kano Memorial Lectures, 1.

\bibitem{shintani}
T.~Shintani, \emph{On evaluation of zeta functions of totally real algebraic
  number fields at non-positive integers}, J. Fac. Sci. Univ. Tokyo Sect. IA
  Math. \textbf{23} (1976), no.~2, 393--417.

\bibitem{siegel1}
C.~L. Siegel, \emph{The volume of the fundamental domain for some infinite
  groups}, Trans. AMS \textbf{39} (1936), 209--218.

\bibitem{siegel2}
C.~L. Siegel, \emph{{Zur Bestimmung des Fundamentalbereiche der unimodularen
  Gruppe}}, Math. Ann. \textbf{137} (1959), 427--432.

\bibitem{pari}
{The PARI~Group}, Bordeaux, \emph{{PARI/GP}}, 2005, available from
  \url{http://pari.math.u-bordeaux.fr/}.

\bibitem{vdg6}
G.~van~der Geer, \emph{Hilbert modular forms for the field {${\mathbb
  Q}(\sqrt6)$}}, Math. Ann. \textbf{233} (1978), no.~2, 163--179.

\bibitem{vdg.minimal}
G.~van~der Geer, \emph{Minimal models for {H}ilbert modular surfaces of
  principal congruence subgroups}, Topology \textbf{18} (1979), no.~1, 29--39.

\bibitem{vdg.vdv}
G.~van~der Geer and A.~van~de Ven, \emph{On the minimality of certain {H}ilbert
  modular surfaces}, Complex analysis and algebraic geometry, Iwanami Shoten,
  Tokyo, 1977, pp.~137--150.

\bibitem{vdgz}
G.~van~der Geer and D.~Zagier, \emph{{The Hilbert modular group for the field
  $\mathbb{Q}(\sqrt{13})$}}, Inv. Math. \textbf{42} (1977), 93--133.

\bibitem{thyang}
T.~Yang, \emph{C{M} number fields and modular forms}, Pure Appl. Math. Q.
  \textbf{1} (2005), no.~2, 305--340.

\end{thebibliography}

\end{document}